\documentclass[12pt]{amsart}

\usepackage{amssymb,amsmath,graphicx,color,textcomp, amsthm,bbm,bbold, enumerate,booktabs}
\usepackage{chngcntr}
\usepackage{apptools}
\usepackage{mathrsfs}
\usepackage{tikz}
\usetikzlibrary{trees}

\AtAppendix{\counterwithin{theorem}{section}}
\definecolor{Red}{cmyk}{0,1,1,0}

\definecolor{verde}{cmyk}{1,0,1,0}

\definecolor{loka}{cmyk}{.5,0,1,.5}
\definecolor{azul}{cmyk}{1,1,0,0}

\numberwithin{equation}{section}

\newcommand{\be}{\begin{equation}}
\newcommand{\ee}{\end{equation}}

\newtheorem{theorem}{Theorem}

\newtheorem{definition}{Definition}
\newtheorem{lemma}{Lemma}
\newtheorem{remark}{Remark}
\newtheorem{corollary}[equation]{Corollary}

\begin{document}
\title{A Gronwall inequality and the Cauchy-type problem by means of $\psi$-Hilfer operator}
\author{J. Vanterler da C. Sousa$^1$}
\address{$^1$ Department of Applied Mathematics, Institute of Mathematics,
 Statistics and Scientific Computation, University of Campinas --
UNICAMP, rua S\'ergio Buarque de Holanda 651,
13083--859, Campinas SP, Brazil\newline
e-mail: {\itshape \texttt{ra160908@ime.unicamp.br, capelas@ime.unicamp.br }}}

\author{E. Capelas de Oliveira$^1$}

\begin{abstract}In this paper, we propose a generalized Gronwall inequality through the fractional integral with respect to another function. The Cauchy-type problem for a nonlinear differential equation involving the $\psi$-Hilfer fractional derivative and the existence and uniqueness of solutions are discussed. Finally, through generalized Gronwall inequality, we prove the continuous dependence of data on the Cauchy-type problem.

\end{abstract}
\maketitle

\section{Introduction}

Over the decades, the fractional calculus has been building a great history and consolidating itself in several scientific areas such as: mathematics, physics and engineering, among others. The emergence of new fractional integrals and derivatives, makes the wide number of definitions becomes increasingly larger and clears its numerous applications \cite{AHMJ,SAM,RHM}. Recently, Sousa and Oliveira \cite{JEM1} introduced the so-called $\psi$-Hilfer fractional derivative with respect to another function, in order to unify the wide number of fractional derivatives in a single fractional operator and consequently, open a window for new applications.

Problems of initial values for the so-called fractional order differential equations that emerge and describe linear and nonlinear phenomena have obtained much attention in the scientific community and especially in engineering. Recently, there are several researchers \cite{PRIN,EDP,EDP1,EDP2,EDP3,EDP4,EDP5,EDP6} that have used fractional differential equations to model natural phenomena. Gronwall inequality is one of the tools used to study the existence, uniqueness and continuous dependence of the Cauchy's problem solutions on data and, in addition, other important applications might be found \cite{DRA}. 

Consider the following Cauchy-type problem for the fractional differential equation \cite{PRIN}
\begin{eqnarray}  \label{ze}
D_{a+}^{\alpha ,\beta }y\left( x\right) &=&f\left( x,y\left(
x\right) \right) \text{, }0<\alpha <1\text{, }0\leq \beta \leq 1 \\
I_{a+}^{1-\gamma}y\left( a\right) &=&y_{a}\text{, \ \ \ \ \ \ \ \ \ \ 
}\gamma =\alpha +\beta \left( 1-\alpha \right),
\end{eqnarray}
where $I_{a+}^{1-\gamma}(\cdot)$ is the fractional integral in the sense of Riemann-Liouville and $D_{a+}^{\alpha ,\beta }(\cdot) $ is the Hilfer fractional derivative.

In this paper, we propose a generalized Gronwall inequality using the fractional integral of a function $f$ with respect to another function $\psi$, and from it to study the continuous dependence of the Cauchy-type problem Eq.(\ref{ze})-Eq.(1.2) on data as well as proposing and discussing the existence and uniqueness by means of $\psi$-Hilfer fractional derivative \cite{JEM1}.

The paper is organized as follows:. in section 2, we begin with the definition of some spaces of functions, the weighted spaces among them. In this section, we present the definition of fractional integral of a function $f$ with respect to another function $\psi$, in a addition to the definitions of $\psi$-Riemann-Liouville fractional derivative, $\psi$-Caputo fractional derivative and $\psi$-Hilfer fractional derivative. From these definitions, some important results are introduced for the development of the paper. Finally, the definition of the Lipschitz condition, is presented. In section 3, we discuss the generalized Gronwall inequality by means of fractional integral with respect to another $\psi$ function. Also, two consonant results of the generalized Gronwall inequality theorem are demonstrated. In section 4, we present the version of the Cauchy-type problem for the $\psi$-Hilfer fractional derivative, showing: this problem is equivalent to solve an Volterra integral equation. From this, we study the existence and uniqueness of Cauchy-type problem. In section 5, using the generalized Gronwall inequality, we address the continuous dependence of the Cauchy-type problem. Concluding remarks close the paper.

\section{Preliminaries}
First, we present the weighted spaces and some definitions and important results for the development of the paper.

Let $[a,b]$ $(0<a<b<\infty)$ be a finite interval on the half-axis $\mathbb{R}^{+}$ and $C[a,b]$, $AC^{n}[a,b]$, $C^{n}[a,b]$ be the spaces of continuous functions, $n$-times absolutely continuous and $n$-times continuously differentiable functions on $[a,b]$, respectively. 

The space of continuous function $f$ on $[a,b]$ with the norm is defined by \cite{AHMJ}
\begin{equation*}
\left\Vert f\right\Vert _{C\left[ a,b\right] }=\underset{t\in \left[ a,b \right] }{\max }\left\vert f\left( t\right) \right\vert.
\end{equation*}

On the other hand, we have $n$-times absolutely continuous functions given by
\begin{equation*}
AC^{n}\left[ a,b\right] =\left\{ f:\left[ a,b\right] \rightarrow \mathbb{R}
;\text{ }f^{\left( n-1\right) }\in AC\left( \left[ a,b\right] \right)
\right\} .
\end{equation*}

The weighted spaces $C_{\gamma,\psi}[a,b]$ and $C_{1-\gamma,\psi}[a,b]$ of functions $f$ on $[a,b]$ are defined by
\begin{equation*}
C_{\gamma ;\psi }\left[ a,b\right] =\left\{ f:\left( a,b\right] \rightarrow 
\mathbb{R};\text{ }\left( \psi \left( t\right) -\psi \left( a\right) \right)
^{\gamma }f\left( t\right) \in C\left[ a,b\right] \right\} ,\text{ }0\leq \gamma <1,
\end{equation*}
\begin{equation*}
C_{1-\gamma ;\psi }\left[ a,b\right] =\left\{ f:\left( a,b\right]
\rightarrow \mathbb{R};\text{ }\left( \psi \left( t\right) -\psi \left(
a\right) \right) ^{1-\gamma }f\left( t\right) \in C\left[ a,b\right]
\right\} ,\text{ }0\leq \gamma <1,
\end{equation*}
respectively, with the norm
\begin{equation}\label{space}
\left\Vert f\right\Vert _{C_{\gamma ;\psi }\left[ a,b\right] }=\left\Vert
\left( \psi \left( t\right) -\psi \left( a\right) \right) ^{\gamma}f\left(
t\right) \right\Vert _{C\left[ a,b\right] }=\underset{t\in \left[ a,b\right] 
}{\max }\left\vert \left( \psi \left( t\right) -\psi \left( a\right) \right)
^{\gamma }f\left( t\right) \right\vert
\end{equation}
and
\begin{equation}\label{space1}
\left\Vert f\right\Vert _{C_{1-\gamma ;\psi }\left[ a,b\right] }=\left\Vert
\left( \psi \left( t\right) -\psi \left( a\right) \right) ^{1-\gamma}f\left(
t\right) \right\Vert _{C\left[ a,b\right] }=\underset{t\in \left[ a,b\right] 
}{\max }\left\vert \left( \psi \left( t\right) -\psi \left( a\right) \right)
^{1-\gamma }f\left( t\right) \right\vert.
\end{equation}

The weighted space $C_{\gamma ;\psi }^{n}\left[ a,b\right]$ of function $f$ on $[a,b]$ is defined by
\begin{equation*}
C_{\gamma;\psi }^{n}\left[ a,b\right] =\left\{ f:\left( a,b\right]
\rightarrow \mathbb{R};\text{ }f\left( t\right) \in C^{n-1}\left[ a,b\right] ;\text{ }f^{\left( n\right) }\left( t\right) \in C_{\gamma;\psi }\left[ a,b\right] \right\} ,\text{ }0\leq \gamma <1
\end{equation*}
with the norm
\begin{equation*}
\left\Vert f\right\Vert _{C_{\gamma ;\psi }^{n}\left[ a,b\right] }=\overset{n-1}{\underset{k=0}{\sum }}\left\Vert f^{\left( k\right) }\right\Vert _{C\left[ a,b\right] }+\left\Vert f^{\left( n\right) }\right\Vert _{C_{\gamma ;\psi }\left[ a,b\right] }.
\end{equation*}

For $n=0$, we have, $C_{\gamma }^{0}\left[ a,b\right] =C_{\gamma }\left[ a,b\right] $.

As one of the objectives of this article is to propose a generalized Gronwall inequality, we present the inequality of Gronwall recently introduced \cite{GRON}. Let $u,$ $v$ be two integrable functions and $g$ a continuous function, with domain $\left[ a,b\right]$. Assume that: $u$ and $v$ are nonnegative; $g$ is nonnegative and nondecreasing. If
\begin{equation*}
u\left( t\right) \leq v\left( t\right) +g\left( t\right) \rho ^{1-\alpha
}\int_{a}^{t}\tau ^{\rho -1}\left( t^{\rho }-\tau ^{\rho }\right) ^{\alpha
-1}u\left( \tau \right) d\tau ,
\end{equation*}%
then
\begin{equation*}
u\left( t\right) \leq v\left( t\right) +\int_{a}^{t}\overset{\infty }{%
\underset{k=1}{\sum }}\frac{\rho ^{1-k\alpha }\left[ g\left( t\right) \Gamma
\left( \alpha \right) \right] ^{k}}{\Gamma \left( \alpha k\right) }\tau
^{\rho -1}\left( t^{\rho }-\tau ^{\rho }\right) ^{k\alpha -1}v\left( \tau
\right) d\tau ,
\end{equation*}
$\forall t\in \left[ a,b\right]$. We suggest \cite{GRON1,GRON2,GRON3}, where other formulations of the Gronwall inequality can be found via fractional integrals.

Let $\alpha>0$, $n\in\mathbb{N}$, with $I=[a,b]$ $(-\infty\leq a<x<b\leq \infty)$ a finite or infinite interval, $f$ an integrable function defined on $I$ and $\psi\in C^{1}([a,b],\mathbb{R})$ an increasing function such that $\psi'(x)\neq 0$, $\forall x\in I$. The fractional integrals of a function $f$ with respect to another function are given by \cite{AHMJ,SAM}
\begin{equation}\label{P1}
I_{a+}^{\alpha ;\psi }f\left( x\right) :=\frac{1}{\Gamma \left( \alpha
\right) }\int_{a}^{x}\psi ^{\prime }\left( t\right) \left( \psi \left(
x\right) -\psi \left( t\right) \right) ^{\alpha -1}f\left( t\right) dt
\end{equation}
and
\begin{equation}\label{P2}
I_{b-}^{\alpha ;\psi }f\left( x\right) :=\frac{1}{\Gamma \left( \alpha
\right) }\int_{x}^{b}\psi ^{\prime }\left( t\right) \left( \psi \left(
t\right) -\psi \left( x\right) \right) ^{\alpha -1}f\left( t\right) dt.
\end{equation}

In this sense, from Eq.(\ref{P1}) and Eq.(\ref{P2}), the Riemann-Liouville fractional derivatives of a function $f$ of order $\alpha$ with respect to another function $\psi$, are defined by \cite{AHMJ,SAM}
\begin{eqnarray}\label{D2}
\mathcal{D}_{a+}^{\alpha ;\psi }f\left( x\right) &=&\left( \frac{1}{\psi ^{\prime
}\left( x\right) }\frac{d}{dx}\right) ^{n}I_{a+}^{n-\alpha ;\psi }f\left(
x\right)  \notag \\
&=&\frac{1}{\Gamma \left( n-\alpha \right) }\left( \frac{1}{\psi ^{\prime
}\left( x\right) }\frac{d}{dx}\right) ^{n}\int_{a}^{x}\psi ^{\prime }\left(
t\right) \left( \psi \left( x\right) -\psi \left( t\right) \right)
^{n-\alpha -1}f\left( t\right) dt \notag \\
\end{eqnarray}
and
\begin{eqnarray}\label{D3}
\mathcal{D}_{b-}^{\alpha ;\psi }f\left( x\right) &=&\left( -\frac{1}{\psi ^{\prime
}\left( x\right) }\frac{d}{dx}\right) ^{n}I_{a+}^{n-\alpha ;\psi }f\left(
x\right)  \notag \\
&=&\frac{1}{\Gamma \left( n-\alpha \right) }\left( -\frac{1}{\psi ^{\prime
}\left( x\right) }\frac{d}{dx}\right) ^{n}\int_{x}^{b}\psi ^{\prime }\left(
t\right) \left( \psi \left( t\right) -\psi \left( x\right) \right)
^{n-\alpha -1}f\left( t\right) dt,\notag \\
\end{eqnarray}
respectively and $n=[\alpha]+1$.

\begin{remark}  From {\rm Eq.(\ref{P1})} and {\rm Eq.(\ref{P2})}, we also consider a fractional derivative version, the so-called $\psi$-Caputo fractional derivative with respect to another function, recently introduced by Almeida {\rm \cite{RCA}}.
\end{remark}

\begin{definition}{\rm \cite{JEM1}} Let $n-1<\alpha <n$ with $n\in\mathbb{N}$, $I=[a,b]$ an interval such that $-\infty\leq a<b\leq\infty$ and $f,\psi\in C^{n}([a,b],\mathbb{R})$ two functions such that $\psi$ is increasing and $\psi'(x)\neq 0$, for all $x\in I$. The  $\psi$-Hilfer fractional derivatives $(\mbox{left-sided and right-sided})$ $^{H}\mathbb{D}^{\alpha,\beta;\psi}_{a+}(\cdot)$ and $^{H}\mathbb{D}^{\alpha,\beta;\psi}_{b-}(\cdot)$ of a function of order $\alpha$ and type $0\leq \beta \leq 1$, are defined by
\begin{equation}\label{HIL}
^{H}\mathbb{D}_{a+}^{\alpha ,\beta ;\psi }f\left( x\right) =I_{a+}^{\beta \left(
n-\alpha \right) ;\psi }\left( \frac{1}{\psi ^{\prime }\left( x\right) }\frac{d}{dx}\right) ^{n}I_{a+}^{\left( 1-\beta \right) \left( n-\alpha
\right) ;\psi }f\left( x\right)
\end{equation}
and
\begin{equation}\label{HIL1}
^{H}\mathbb{D}_{b-}^{\alpha ,\beta ;\psi }f\left( x\right) =I_{b-}^{\beta
\left( n-\alpha \right) ;\psi }\left( -\frac{1}{\psi ^{\prime }\left(
x\right) }\frac{d}{dx}\right) ^{n}I_{b-}^{\left( 1-\beta \right) \left(
n-\alpha \right) ;\psi }f\left( x\right).
\end{equation}

The $\psi$-Hilfer fractional derivatives as above defined, can be written in the following forms
\begin{equation}\label{HIL2}
^{H}\mathbb{D}_{a+}^{\alpha ,\beta ;\psi }f\left( x\right) =I_{a+}^{\gamma -\alpha ;\psi }\mathcal{D}_{a+}^{\gamma ;\psi }f\left( x\right) 
\end{equation}
and
\begin{equation}\label{HIL3}
^{H}\mathbb{D}_{b-}^{\alpha ,\beta ;\psi }f\left( x\right) =I_{b-}^{\gamma -\alpha ;\psi }\left( -1\right) ^{n}\mathcal{D}_{b-}^{\gamma ;\psi }f\left( x\right),
\end{equation}
respectively with $\gamma =\alpha +\beta \left( n-\alpha \right) $ and $I^{\gamma-\alpha;\psi}_{a+}(\cdot)$, $D^{\gamma;\psi}_{a+}(\cdot)$, $I^{\gamma-\alpha;\psi}_{b-}(\cdot)$, $D^{\gamma;\psi}_{b-}(\cdot)$ as defined in {\rm Eq.(\ref{P1})}, {\rm Eq.(\ref{D2})}, {\rm Eq.(\ref{P2})} and {\rm Eq.(\ref{D3})}.
\end{definition}

\begin{definition}{\rm \cite{PRIN}} Assume that $f\left( x,y\left( x\right) \right) $ is defined on set $\left( a,b\right] \times H,$ $H\subset  \mathbb{R} .$ A function $f\left( x,y\left( x\right) \right) $ satisfies Lipschitz condition with respect to $y$, if all $x\in \left( a,b\right] $ and for $y_{1},y_{2}\in H,$
\begin{equation}\label{eq10}
\left\vert f\left( x,y_{1}\right) -f\left( x,y_{2}\right) \right\vert \leq
A\left\vert y_{1}-y_{2}\right\vert,
\end{equation}
where $A>0$ is Lipschitz constant.
\end{definition}

\begin{definition} Let $0<\alpha <1,$ $0\leq \beta \leq 1,$ the weighted space $C_{1-\gamma ;\psi }^{\alpha ,\beta }\left[ a,b\right] $ is defined by 
\begin{equation*}
C_{1-\gamma ;\psi }^{\alpha ,\beta }\left[ a,b\right] =\left\{ f\in C_{1-\gamma ;\psi }\left[ a,b\right] :\text{ }^{H}\mathbb{D}_{a+}^{\alpha ,\beta ;\psi }f\in C_{1-\gamma ;\psi }\left[ a,b\right] \right\},
\end{equation*}
with $\gamma =\alpha +\beta \left( 1-\alpha \right)$.
\end{definition}

\begin{lemma}\label{lema1} If $\alpha >0$ and $0\leq \mu <1,$ then $I_{a+}^{\alpha ;\psi }(\cdot)$ is bounded from $C_{\mu ;\psi }\left[ a,b\right] $ to $C_{\mu ;\psi }\left[ a,b\right] .$ In addition, if $\mu \leq \alpha $, then $I_{a+}^{\alpha ;\psi }(\cdot)$ is bounded from $C_{\mu ;\psi }\left[ a,b\right] $ to $C\left[ a,b\right] $.
\end{lemma}

\begin{lemma} \label{lema2} Let $\alpha>0$ and $\delta>0$.
\begin{enumerate}
\item If $f(x)= \left( \psi \left( x\right) -\psi \left( a\right)
\right) ^{\delta -1}$, then 
\begin{equation*}
I_{a+}^{\alpha ;\psi }f(x)=\frac{\Gamma \left( \delta \right) }{\Gamma \left(
\alpha +\delta \right) }\left( \psi \left( x\right) -\psi \left( a\right)
\right) ^{\alpha +\delta -1}
\end{equation*}

\item If $g(x)=\left( \psi \left( b\right) -\psi \left( x\right)
\right) ^{\delta -1}$, then
\begin{equation*}
I_{b-}^{\alpha ;\psi }g(x)=\frac{\Gamma \left( \delta \right) }{\Gamma\left(
\alpha +\delta \right) }\left( \psi \left( b\right) -\psi \left( x\right)
\right) ^{\alpha +\delta -1}
\end{equation*}

\end{enumerate}
\end{lemma}

\begin{proof}
See {\rm \cite{AHMJ}}.
\end{proof}

\begin{theorem}\label{teo1} If $f\in C^{n}[a,b]$, $n-1<\alpha<n$ and $0\leq \beta \leq 1$, then
\begin{equation*}
I_{a+}^{\alpha ;\psi }\text{ }^{H}\mathbb{D}_{a+}^{\alpha ,\beta ;\psi }f\left( x\right) =f\left( x\right) -\overset{n}{\underset{k=1}{\sum }}\frac{\left( \psi \left( x\right) -\psi \left( a\right) \right) ^{\gamma -k}}{\Gamma \left( \gamma -k+1\right) }f_{\psi }^{\left[ n-k\right] }I_{a+}^{\left( 1-\beta \right) \left( n-\alpha \right) ;\psi }f\left( a\right) 
\end{equation*}
and 
\begin{equation*}
I_{b-}^{\alpha ;\psi }\text{ }^{H}\mathbb{D}_{b-}^{\alpha ,\beta ;\psi }f\left( x\right) =f\left( x\right) -\overset{n}{\underset{k=1}{\sum }}\frac{\left( -1\right) ^{k}\left( \psi \left( b\right) -\psi \left( x\right) \right) ^{\gamma -k}}{\Gamma \left( \gamma -k+1\right) }f_{\psi }^{\left[ n-k\right] }I_{b-}^{\left( 1-\beta \right) \left( n-\alpha \right) ;\psi }f\left(
b\right) .
\end{equation*}
\end{theorem}
\begin{proof}
See {\rm \cite{JEM1}}.
\end{proof}

\begin{theorem}\label{teo2} Let $f\in C^{1}[a,b]$, $\alpha>0$ and $0\leq\beta\leq 1$, we have
\begin{equation*}
^{H}\mathbb{D}_{a+}^{\alpha ,\beta ;\psi }I_{a+}^{\alpha ;\psi }f\left( x\right)
=f\left( x\right) \text{ and }^{H}\mathbb{D}_{b-}^{\alpha ,\beta ;\psi
}I_{b-}^{\alpha ;\psi }f\left( x\right) =f\left( x\right) .
\end{equation*}
\end{theorem}

\begin{proof}
See {\rm \cite{JEM1}}.
\end{proof}

\section{The Gronwall inequality}

The Gronwall inequality plays an important role in the study of the qualitative theory of integral and differential equations \cite{PRIN,GRON,WAN,DANI}, as well as solving Cauchy-type problems of non-linear differential equations. In order to work with continuous dependence of differential equations via $\psi$-Hilfer fractional derivative, in this section, we present the first main result, the generalized Gronwall inequality by means of the fractional integral with respect to another function $\psi$ and other important results.

\begin{theorem}\label{teo3} Let $u,$ $v$ be two integrable functions and $g$ continuous, with domain $\left[ a,b\right] .$ Let $\psi \in C^{1}\left[ a,b\right] $ an increasing function such that $\psi ^{\prime }\left( t\right) \neq 0$, $\forall t\in \left[ a,b\right]$. Assume that
\begin{enumerate}
\item $u$ and $v$ are nonnegative;
\item $g$ in nonnegative and nondecreasing.
\end{enumerate}

If
\begin{equation*}
u\left( t\right) \leq v\left( t\right) +g\left( t\right) \int_{a}^{t}\psi
^{\prime }\left( \tau \right) \left( \psi \left( t\right) -\psi \left( \tau
\right) \right) ^{\alpha -1}u\left( \tau \right) d\tau,
\end{equation*}
then
\begin{equation}\label{jose}
u\left( t\right) \leq v\left( t\right) +\int_{a}^{t}\overset{\infty }{%
\underset{k=1}{\sum }}\frac{\left[ g\left( t\right) \Gamma \left( \alpha
\right) \right] ^{k}}{\Gamma \left( \alpha k\right) }\psi ^{\prime }\left(
\tau \right) \left[ \psi \left( t\right) -\psi \left( \tau \right) \right]
^{\alpha k-1}v\left( \tau \right) d\tau,
\end{equation}
$\forall t\in \left[ a,b\right]$.
\end{theorem}

\begin{proof}
Let
\begin{equation}  \label{A}
A\phi \left( t\right) =g\left( t\right) \int_{a}^{t}\psi ^{\prime }\left(
\tau \right) \left( \psi \left( t\right) -\psi \left( \tau \right) \right)
^{\alpha -1}\phi \left( \tau \right) d\tau,
\end{equation}
$\forall t\in \left[ a,b\right] ,$ for locally integral functions $\phi$. Then,
\begin{equation*}
u\left( t\right) \leq v\left( t\right) +Au\left( t\right).
\end{equation*}

Iterating, for $n\in \mathbb{N}$, we can write
\begin{equation*}
u\left( t\right) \leq \underset{k=0}{\overset{n-1}{\sum }}A^{k}v\left(
t\right) +A^{n}u\left( t\right).
\end{equation*}

The next step, is usually mathematical induction, that if $\phi $ is a nonnegative function, then
\begin{equation} \label{eq11}
A^{k}u\left( t\right) \leq \int_{a}^{t}\frac{\left[ g\left( t\right) \Gamma
\left( \alpha \right) \right] ^{k}}{\Gamma \left( \alpha k\right) }\psi
^{\prime }\left( \tau \right) \left[ \psi \left( t\right) -\psi \left( \tau
\right) \right] ^{\alpha k-1}u\left( \tau \right) d\tau.
\end{equation}

We know that relation {\rm Eq.(\ref{eq11})} is true for $n=1.$ Suppose that the formula is true for some $k=n\in \mathbb{N}$, then the induction hypothesis implies
\begin{eqnarray}\label{eq12}
A^{k+1}u\left( t\right) = &&A\left( A^{k}u\left( t\right) \right) \leq
A\left( \int_{a}^{t}\frac{\left[ g\left( t\right) \Gamma \left( \alpha
\right) \right] ^{k}}{\Gamma \left( \alpha k\right) }\psi ^{\prime }\left(
\tau \right) \left[ \psi \left( t\right) -\psi \left( \tau \right) \right]
^{\alpha k-1}u\left( \tau \right) d\tau \right)  \notag  \label{C} \\
= &&g\left( t\right) \int_{a}^{t}\psi ^{\prime }\left( \tau \right) \left[
\psi \left( t\right) -\psi \left( \tau \right) \right] ^{\alpha -1}\left(
\int_{a}^{\tau }\frac{\left[ g\left( \tau \right) \Gamma \left( \alpha
\right) \right] ^{k}}{\Gamma \left( \alpha k\right) }\psi ^{\prime }\left(
s\right) \left[ \psi \left( \tau \right) -\psi \left( s\right) \right]
^{\alpha k-1}u\left( s\right) ds\right) d\tau. \notag \\
\end{eqnarray}

By hypothesis, $g$ is a nondecreasing function, that is $g\left( \tau \right) \leq g\left( t\right) ,$ for all $\tau \leq t$, then from {\rm Eq.(\ref{eq12})}, we get
\begin{equation}  \label{eq13}
A^{k+1}u\left( t\right) \leq \frac{\Gamma \left( \alpha \right) ^{k}}{\Gamma
\left( \alpha k\right) }\left[ g\left( t\right) \right] ^{k+1}\int_{a}^{t}%
\int_{a}^{\tau }\psi ^{\prime }\left( \tau \right) \left[ \psi \left(
t\right) -\psi \left( \tau \right) \right] ^{\alpha -1}\psi ^{\prime }\left(
s\right) \left[ \psi \left( \tau \right) -\psi \left( s\right) \right]
^{\alpha k-1}u\left( s\right) dsd\tau.
\end{equation}

By Dirichlet's formula, the {\rm Eq.(\ref{eq13})} can be rewritten as
\begin{equation}\label{eq14}
A^{k+1}u\left( t\right) \leq \frac{\Gamma \left( \alpha \right) ^{k}}{\Gamma
\left( \alpha k\right) }\left[ g\left( t\right) \right] ^{k+1}\int_{a}^{t}%
\psi ^{\prime }\left( \tau \right) u\left( \tau \right) \int_{\tau }^{t}\psi
^{\prime }\left( s\right) \left[ \psi \left( t\right) -\psi \left( s\right) %
\right] ^{\alpha -1}\left[ \psi \left( s\right) -\psi \left( \tau \right) %
\right] ^{\alpha k-1}dsd\tau .
\end{equation}

Note that,
\begin{eqnarray*}
&&\int_{\tau }^{t}\psi ^{\prime }\left( s\right) \left[ \psi \left( t\right)
-\psi \left( s\right) \right] ^{\alpha -1}\left[ \psi \left( s\right) -\psi
\left( \tau \right) \right] ^{\alpha k-1}ds \\
&=&\int_{\tau }^{t}\psi ^{\prime }\left( s\right) \left[ \psi \left(
t\right) -\psi \left( s\right) \right] ^{\alpha -1}\left[ 1-\frac{\psi
\left( s\right) -\psi \left( \tau \right) }{\psi \left( t\right) -\psi
\left( \tau \right) }\right] ^{\alpha -1}\left[ \psi \left( s\right) -\psi
\left( \tau \right) \right] ^{\alpha k-1}ds.
\end{eqnarray*}

Introducing a change of variables $u=\dfrac{\psi \left( s\right) -\psi \left( \tau \right) }{\psi \left( t\right) -\psi \left( \tau \right) }$ and using the definition of beta function and the relation with gamma function $B\left( x,y\right) =\dfrac{\Gamma \left( x\right) \Gamma \left( y\right) }{\Gamma \left( x+y\right) }$, we have
\begin{eqnarray}\label{eq15}
&&\int_{\tau }^{t}\psi ^{\prime }\left( s\right) \left[ \psi \left( t\right)
-\psi \left( s\right) \right] ^{\alpha -1}\left[ \psi \left( s\right) -\psi
\left( \tau \right) \right] ^{\alpha k-1}ds  \notag \\
&=&\left[ \psi \left( t\right) -\psi \left( \tau \right) \right] ^{k\alpha
+\alpha -1}\int_{0}^{1}\left[ 1-u\right] ^{\alpha -1}u^{k\alpha -1}du  \notag
\\
&=&\left[ \psi \left( t\right) -\psi \left( \tau \right) \right] ^{k\alpha
+\alpha -1}\frac{\Gamma \left( \alpha \right) \Gamma \left( k\alpha \right) 
}{\Gamma \left( \alpha +k\alpha \right) }.
\end{eqnarray}

Replacing {\rm Eq.(\ref{eq14})} in {\rm Eq.(\ref{eq15})}, we get
\begin{equation*}
A^{k+1}u\left( t\right) \leq \int_{a}^{t}\frac{\left[ g\left( t\right)
\Gamma \left( \alpha \right) \right] ^{k+1}}{\Gamma \left( \alpha \left(
k+1\right) \right) }\psi ^{\prime }\left( \tau \right) u\left( \tau \right) %
\left[ \psi \left( t\right) -\psi \left( \tau \right) \right] ^{\alpha
\left( k+1\right) -1}d\tau.
\end{equation*}

Let us now prove that $A^{n}u\left( t\right) \rightarrow 0$ as $n\rightarrow \infty .$ As $g$ is a continuous function on $\left[ a,b\right] $, then by Weierstrass theorem {\rm \cite{COU,ELON}}, there exist a constant $M>0$ such that $g\left( t\right) \leq M$ for all $t\in \left[ a,b\right] $. Then, we obtain
\begin{equation*}
A^{n}u\left( t\right) \leq \int_{a}^{t}\frac{\left[ M\Gamma \left( \alpha
\right) \right] ^{n}}{\Gamma \left( \alpha n\right) }\psi ^{\prime }\left(
\tau \right) u\left( \tau \right) \left[ \psi \left( t\right) -\psi \left(
\tau \right) \right] ^{\alpha n-1}d\tau.
\end{equation*}

Consider the series
\begin{equation*}
\overset{\infty }{\underset{n=1}{\sum }}\frac{\left[ M\Gamma \left( \alpha
\right) \right] ^{n}}{\Gamma \left( \alpha n\right) },
\end{equation*}
satisfying the relation
\begin{equation}
\underset{n\rightarrow \infty }{\lim }\frac{\Gamma \left( \alpha n\right)
\left( \alpha n\right) ^{\alpha }}{\Gamma \left( \alpha n+\alpha \right) }=1.
\end{equation}

Using the ratio test to the series and the asymptotic approximation {\rm \cite{WONR}}, so we obtain
\begin{equation*}
\underset{n\rightarrow \infty }{\lim }\frac{\Gamma \left( \alpha n\right) }{%
\Gamma \left( \alpha n+\alpha \right) }=0.
\end{equation*}

Therefore, the series converges and we conclude that
\begin{equation*}
u\left( t\right) \leq \overset{\infty }{\underset{k=0}{\sum }}A^{k}v\left(
t\right) \leq v\left( t\right) +\int_{a}^{t}\underset{k=1}{\overset{\infty }{%
\sum }}\frac{\left[ g\left( t\right) \Gamma \left( \alpha \right) \right]
^{k}}{\Gamma \left( \alpha k\right) }\psi ^{\prime }\left( \tau \right) %
\left[ \psi \left( t\right) -\psi \left( \tau \right) \right] ^{\alpha
k-1}v\left( \tau \right) d\tau.
\end{equation*}
\end{proof}

\begin{corollary} Let $\alpha>0$, $I=[a,b]$ and $f,\psi\in C^{1}([a,b],\mathbb{R})$ two functions such that $\psi$ is increasing and $\psi'(t)\neq 0$ for all $t\in I$. Suppose $b\geq 0$ and $v$ is a nonnegative function locally integrable on $\left[ a,b\right] $ and suppose also that $u$ is nonnegative and locally integrable on $\left[ a,b\right] $ with 
\begin{equation*}
u\left( t\right) \leq v\left( t\right) +b\int_{a}^{t}\psi ^{\prime }\left(
\tau \right) \left[ \psi \left( t\right) -\psi \left( \tau \right) \right]
^{\alpha -1}u\left( \tau \right) d\tau, \text{ }\forall t\in[a,b].
\end{equation*}

Then, we can write
\begin{equation*}
u\left( t\right) \leq v\left( t\right) +\int_{a}^{t}\underset{k=1}{\overset{%
\infty }{\sum }}\frac{\left[ b\Gamma \left( \alpha \right) \right] ^{k}}{%
\Gamma \left( \alpha k\right) }\psi ^{\prime }\left( \tau \right) \left[
\psi \left( t\right) -\psi \left( \tau \right) \right] ^{\alpha k-1}v\left(
\tau \right) d\tau, \text{ }\forall t\in[a,b].
\end{equation*}
\end{corollary}

\begin{corollary} Under the hypothesis of {\rm Theorem \ref{teo3}}, let $v$ be a nondecreasing function on $\left[ a,b\right] $. Then, we have
\begin{equation*}
u\left( t\right) \leq v\left( t\right) \mathbb{E}_{\alpha }\left( g\left( t\right)
\Gamma \left( \alpha \right) \left[ \psi \left( t\right) -\psi \left( \tau
\right) \right] ^{\alpha }\right) ,\text{ }\forall t\in \left[ a,b\right],
\end{equation*}
where $\mathbb{E}_{\alpha }\left( \cdot \right) $ is the Mittag-Leffler function defined by $\mathbb{E}_{\alpha }\left( t\right) =\underset{k=0}{\overset{\infty }{\sum }}\dfrac{t^{k}}{\Gamma \left( \alpha k+1\right) }$ with $Re(\alpha)>0$.
\end{corollary}

\begin{proof} In fact, as $v$ is nondecreasing, so, for all $\tau \in \left[ a,t \right] $, we have $v\left( \tau \right) \leq v\left( t\right) $ and we can write
\begin{eqnarray*}
u\left( t\right) &\leq &v\left( t\right) +\int_{a}^{t}\overset{\infty }{%
\underset{k=1}{\sum }}\frac{\left[ g\left( t\right) \Gamma \left( \alpha
\right) \right] ^{k}}{\Gamma \left( \alpha k\right) }\psi ^{\prime }\left(
\tau \right) \left[ \psi \left( t\right) -\psi \left( \tau \right) \right]
^{\alpha k-1}v\left( \tau \right) d\tau  \notag \\
&=&v\left( t\right) \left[ 1+\int_{a}^{t}\overset{\infty }{\underset{k=1}{%
\sum }}\frac{\left[ g\left( t\right) \Gamma \left( \alpha \right) \right]
^{k}}{\Gamma \left( \alpha k\right) }\psi ^{\prime }\left( \tau \right) %
\left[ \psi \left( t\right) -\psi \left( \tau \right) \right] ^{\alpha
k-1}v\left( \tau \right) d\tau \right]  \notag \\
&=&v\left( t\right) \left[ 1+\overset{\infty }{\underset{k=1}{\sum }}\frac{%
\left[ g\left( t\right) \Gamma \left( \alpha \right) \right] ^{k}}{\Gamma
\left( \alpha k\right) }\frac{\left[ \psi \left( t\right) -\psi \left( \tau
\right) \right] ^{\alpha k}}{k\alpha }\right]  \notag \\
&=&v\left( t\right) \mathbb{E}_{\alpha }\left( g\left( t\right) \Gamma \left( \alpha
\right) \left[ \psi \left( t\right) -\psi \left( \tau \right) \right]
^{\alpha }\right).
\end{eqnarray*}
\end{proof}

\begin{remark}
Recently, Sousa and Oliveira {\rm \cite{JEM1}}, introduced $\psi$-Hilfer fractional derivative, and presented a class of fractional integrals and fractional derivatives, which can be obtained from such definitions, with the particular choice of the function $\psi$. In this sense, the Gronwall inequality given above can be considered as a class of inequality of Gronwall, that is, in other words, it's a generalization of the possible inequalities of Gronwall introduced via fractional integral. We suggest {\rm \cite{GRON1,GRON2,GRON3}}, where some Gronwall inequalities via fractional integral are present and discussed.
\end{remark}

\section{Existence and uniqueness}
In this section, we prove the second main result of the paper, the existence and uniqueness of solutions of the Cauchy-type problem Eq.(\ref{IVPI})-Eq.(4.2), using the fact that Volterra integral equation is equivalent to the Cauchy-type problem in the $C_{1-\gamma;\psi}[a,b]$ weighted space, by means of $\psi$-Hilfer fractional derivative.

To this end, we consider the following initial value problem
\begin{eqnarray}  \label{IVPI}
^{H}\mathbb{D}_{a+}^{\alpha ,\beta ;\psi }y\left( x\right) &=&f\left( x,y\left(
x\right) \right) \text{, }0<\alpha <1,0\leq \beta \leq 1 \\
I_{a+}^{1-\gamma ;\psi }y\left( a\right) &=&y_{a}\text{, \ \ \ \ \ \ \ \ \ \ 
}\gamma =\alpha +\beta \left( 1-\alpha \right),
\end{eqnarray}
where $^{H}\mathbb{D}_{a+}^{\alpha ,\beta ;\psi }y\left( x\right) $ is the $\psi -$Hilfer fractional derivative, $f:\left[ a,b\right) \times \mathbb{R}\rightarrow  \mathbb{R}$ and $y_{a}$ is a constant.

Applying the fractional integral operator $I_{a+}^{\alpha ;\psi }\left( \cdot \right) $ on both sides of the fractional equation Eq.(\ref{IVPI}) and using Theorem \ref{teo1}, we get
\begin{equation}\label{eq19}
y\left( x\right) =\frac{\left( \psi \left( x\right) -\psi \left( a\right) \right) ^{\gamma -1}}{\Gamma \left( \gamma \right) }I_{a+}^{\left( 1-\beta \right) \left( 1-\alpha \right) ;\psi }f\left( a\right) +I_{a+}^{\alpha ;\psi }f\left( x,y\left( x\right) \right).
\end{equation}

On the other hand, if $y$ satisfies Eq.(\ref{eq19}), then $y$ satisfies Eq.(\ref{IVPI})-Eq.(4.2). However, applying the fractional derivative operator $^{H}\mathbb{D}_{a+}^{\alpha ,\beta ;\psi }\left( \cdot \right) $ on both sides of Eq.(\ref{eq19}), we have
\begin{equation}
^{H}\mathbb{D}_{a+}^{\alpha ,\beta ;\psi }y\left( x\right) =\text{ } ^{H}\mathbb{D}_{a+}^{\alpha ,\beta ;\psi }\left( \frac{\left( \psi \left( x\right) -\psi \left( a\right) \right) ^{\gamma -1}}{\Gamma \left( \gamma \right) } I_{a+}^{\left( 1-\beta \right) \left( 1-\alpha \right) ;\psi }f\left( a\right) \right) +^{H}\mathbb{D}_{a+}^{\alpha ,\beta ;\psi }I_{a+}^{\alpha ;\psi
}f\left( x,y\left( x\right) \right).
\end{equation}

Using the Theorem {\ref{teo2} and the formula \cite{JEM1,EDP2}
\begin{equation*}
^{H}\mathbb{D}_{a+}^{\alpha ,\beta ;\psi }\left( \psi \left( x\right) -\psi \left(a\right) \right) ^{\gamma -1}=0,\text{ }0<\gamma <1,
\end{equation*}
we obtain
\begin{equation*}
^{H}\mathbb{D}_{a+}^{\alpha ,\beta ;\psi }y\left( x\right) =f\left( x,y\left(x\right) \right).
\end{equation*}

Then, we conclude that, $y\left( x\right) $ satisfies initial value problem Eq.(\ref{IVPI})-Eq.(4.2) if and only if $y\left( x\right) $ satisfies the Volterra integral equation of second kind
\begin{equation}\label{eq21}
y\left( x\right) =y_{a}\frac{\left( \psi \left( x\right) -\psi \left( a\right) \right) ^{\gamma -1}}{\Gamma \left( \gamma \right) }+\frac{1}{\Gamma \left( \alpha \right) }\int_{a}^{x}\psi ^{\prime }\left( t\right) \left( \psi \left( x\right) -\psi \left( t\right) \right) ^{\alpha
-1}f\left( t,y\left( t\right) \right) dt.
\end{equation}

\begin{lemma}\label{lema3} Let $\psi\in C^{1}([a,b],\mathbb{R})$ be a function such that $\psi$  is increasing and $\psi\neq 0$, $\forall x\in[a,b]$. If $\gamma =\alpha +\beta \left( 1-\alpha \right) $ where $ 0<\alpha <1$ and $0\leq \beta \leq 1,$ then $\psi$-Riemann fractional integral operator $I_{a+}^{\alpha ;\psi }\left( \cdot \right) $ is bounded from $C_{1-\gamma ;\psi }\left[ a,b\right] $ to $C_{1-\gamma ;\psi }\left[
a,b\right]$:
\begin{equation}\label{eq23}
\left\Vert I_{a+}^{\alpha ;\psi }f\right\Vert _{C_{1-\gamma ;\psi }\left[ a,b \right] }\leq M\frac{\Gamma \left( \gamma \right) }{\Gamma \left( \gamma +\alpha \right) }\left( \psi \left( x\right) -\psi \left( a\right) \right)^{\alpha },
\end{equation}
where, $M$ is the bound of a bounded function $f$.
\end{lemma}

\begin{proof} From {\rm Lemma \ref{lema1}}, the result follows. Now we prove the estimative {\rm Eq.(\ref{eq23})}. By the weighted space given in {\rm Eq.(\ref{space})}, we have
\begin{eqnarray*}
\left\Vert I_{a+}^{\alpha ;\psi }f\right\Vert _{C_{1-\gamma ;\psi }\left[ a,b%
\right] } &=&\left\Vert \left( \psi \left( x\right) -\psi \left( a\right)
\right) ^{1-\gamma }I_{a+}^{\alpha ;\psi }f\right\Vert _{C\left[ a,b\right] }
\\
&\leq &\left\Vert \left( \psi \left( x\right) -\psi \left( a\right) \right)
^{1-\gamma }f\right\Vert _{C\left[ a,b\right] } \\
&&\left\Vert \left( \psi \left( x\right) -\psi \left( a\right) \right)
^{1-\gamma }I_{a+}^{\alpha ;\psi }\left( \psi \left( x\right) -\psi \left(
a\right) \right) ^{\gamma -1}\right\Vert _{C\left[ a,b\right] } \\
&=&\left\Vert f\right\Vert _{C_{1-\gamma ;\psi }\left[ a,b\right]
}\left\Vert I_{a+}^{\alpha ;\psi }\left( \psi \left( x\right) -\psi \left(
a\right) \right) ^{\gamma -1}\right\Vert _{C_{1-\gamma ;\psi }\left[ a,b%
\right] } \\
&\leq &M\frac{\Gamma \left( \gamma \right) }{\Gamma \left( \gamma +\alpha
\right) }\frac{\left( \psi \left( x\right) -\psi \left( a\right) \right)
^{\gamma }}{\psi \left( x\right) -\psi \left( a\right) }\left( \psi \left(
x\right) -\psi \left( a\right) \right) ^{\alpha }.
\end{eqnarray*}

As $0<\gamma <1,$ then $\dfrac{\left( \psi \left( x\right) -\psi \left( a\right) \right) ^{\gamma }}{\psi \left( x\right) -\psi \left( a\right) }<1$. So, we conclude that
\begin{equation*}
\left\Vert I_{a+}^{\alpha ;\psi }f\right\Vert _{C_{1-\gamma ;\psi }\left[ a,b \right] }\leq M\frac{\Gamma \left( \gamma \right) }{\Gamma \left( \gamma +\alpha \right) }\left( \psi \left( x\right) -\psi \left( a\right) \right) ^{\alpha }.
\end{equation*}
\end{proof}

\begin{theorem}Let $\gamma =\alpha +\beta \left( 1-\alpha \right) $ where $0<\alpha <1$ and $0\leq \beta \leq 1.$ Let $f:\left[ a,b\right] \times \mathbb{R}\rightarrow \mathbb{R}
$ be a function such that $f\left( x,y\right) \in C_{1-\gamma ;\psi }\left[ a,b\right] $ for any $y\in C_{1-\gamma ;\psi }\left[ a,b\right] $ and satisfies Lipschitz condition {\rm Eq.(\ref{eq10})} with respect to $y$. Then there exists a unique solution $y\left( x\right) $ for the Cauchy-type problem {\rm Eq.(\ref{IVPI})-Eq.(4.2)} $ C_{1-\gamma ;\psi }^{\alpha ,\beta }\left[ a,b\right]$.
\end{theorem}

\begin{proof} The Volterra integral equation makes sense in any interval $\left[ a,x_{1}\right] \subset \left[ a,b\right]$. So, we choose $x_{1}$ such that
\begin{equation}\label{32}
A\frac{\Gamma \left( \gamma \right) }{\Gamma \left( \gamma +\alpha \right) }%
\left( \psi \left( x\right) -\psi \left( a\right) \right) ^{\alpha }<1,
\end{equation}
holds and first we prove the existence of unique solution $y\in C_{1-\gamma ;\psi }\left[ a,x_{1}\right]$. From the Picard's sequences, we consider
\begin{equation}\label{33}
y_{0}\left( x\right) =y_{a}\frac{\left( \psi \left( x\right) -\psi \left(
a\right) \right) ^{\gamma -1}}{\Gamma \left( \gamma \right) },\text{ }\gamma
=\alpha +\beta \left( 1-\alpha \right) 
\end{equation}
and
\begin{equation}\label{34}
y_{m}\left( x\right) =y_{0}\left( x\right) +\frac{1}{\Gamma \left( \alpha
\right) }\int_{a}^{x}\psi ^{\prime }\left( t\right) \left( \psi \left(
x\right) -\psi \left( t\right) \right) ^{\alpha -1}f\left( t,y_{m-1}\left(
t\right) \right) dt,\text{ }m\in \mathbb{N}.
\end{equation}

We now show that $y_{m}\left( x\right) \in C_{1-\gamma ;\psi }\left[ a,b\right] .$ From {\rm Eq.(\ref{33})}, it follows that $y_{0}\left( x\right) \in C_{1-\gamma ;\psi }\left[ a,b\right]$. By {\rm Lemma \ref{lema1}} $I_{a+}^{\alpha ;\psi }f $ is bounded from $C_{1-\gamma ;\psi }\left[ a,b\right] $ to $C_{1-\gamma ;\psi }\left[ a,b\right] $, which gives $y_{m}\left( x\right) \in C_{1-\gamma ;\psi }\left[ a,b\right] $, $m\in  \mathbb{N} $ and $x\in \left[ a,b\right]$.

Using {\rm Eq.(\ref{33})} and {\rm Eq.(\ref{34})}, we have
\begin{eqnarray*}
&&\left\Vert y_{1}\left( x\right) -y_{0}\left( x\right) \right\Vert
_{C_{1-\gamma ;\psi }\left[ a,x_{1}\right] } \\
&=&\left\Vert y_{0}\left( x\right) +\frac{1}{\Gamma \left( \alpha \right) }%
\int_{a}^{x}\psi ^{\prime }\left( t\right) \left( \psi \left( x\right) -\psi
\left( t\right) \right) ^{\alpha -1}f\left( t,y_{0}\left( t\right) \right)
dt-y_{0}\left( x\right) \right\Vert _{C_{1-\gamma ;\psi }\left[ a,x_{1}%
\right] } \\
&=&\left\Vert I_{a+}^{\alpha ;\psi }f\left( x,y_{0}\left( x\right) \right)
\right\Vert _{C_{1-\gamma ;\psi }\left[ a,x_{1}\right] }.
\end{eqnarray*}

By means of {\rm Lemma \ref{lema3}}, we have 
\begin{equation}\label{35}
\left\Vert y_{1}\left( x\right) -y_{0}\left( x\right) \right\Vert _{C_{1-\gamma ;\psi }\left[ a,x_{1}\right] }\leq M\frac{\Gamma \left( \gamma \right) }{\Gamma \left( \gamma +\alpha \right) }\left( \psi \left( x\right) -\psi \left( a\right) \right) ^{\alpha }.
\end{equation}

Further,we get 
\begin{eqnarray*}
\left\Vert y_{2}\left( x\right) -y_{1}\left( x\right) \right\Vert
_{C_{1-\gamma ;\psi }\left[ a,x_{1}\right] } &=&\left\Vert I_{a+}^{\alpha
;\psi }\left[ f\left( x,y\left( x\right) \right) -f\left( x,y_{0}\left(
x\right) \right) \right] \right\Vert _{C_{1-\gamma ;\psi }\left[ a,x_{1}%
\right] }  \notag  \label{36} \\
&\leq &\frac{\Gamma \left( \gamma \right) \left( \psi \left( x_{1}\right)
-\psi \left( a\right) \right) ^{\alpha }}{\Gamma \left( \gamma +\alpha
\right) }A\left\Vert y\left( x\right) -y_{0}\left( x\right) \right\Vert
_{C_{1-\gamma ;\psi }\left[ a,x_{1}\right] }  \notag \\
&\leq &M\frac{\Gamma \left( \gamma \right) \left( \psi \left( x_{1}\right)
-\psi \left( a\right) \right) ^{\alpha }}{\Gamma \left( \gamma +\alpha
\right) }\left( A\frac{\Gamma \left( \gamma \right) \left( \psi \left(
x_{1}\right) -\psi \left( a\right) \right) ^{\alpha }}{\Gamma \left( \gamma
+\alpha \right) }\right).  
\end{eqnarray*}

Continuing this process $m$-times, we can write
\begin{equation}\label{37}
\left\Vert y_{m}\left( x\right) -y_{m-1}\left( x\right) \right\Vert _{C_{1-\gamma ;\psi }\left[ a,x_{1}\right] }\leq M\frac{\Gamma \left( \gamma \right) \left( \psi \left( x_{1}\right) -\psi \left( a\right) \right)
^{\alpha }}{\Gamma \left( \gamma +\alpha \right) }\left( A\frac{\Gamma \left( \gamma \right) \left( \psi \left( x_{1}\right) -\psi \left( a\right)
\right) ^{\alpha }}{\Gamma \left( \gamma +\alpha \right) }\right) ^{m-1}.
\end{equation}

Taking the limit $m\rightarrow \infty $ on both sides of {\rm Eq.(\ref{37})} and remember the following condition $A\dfrac{ \Gamma \left( \gamma \right) \left( \psi \left( x_{1}\right) -\psi \left( a\right) \right) ^{\alpha }}{\Gamma \left( \gamma +\alpha \right) }<1$, we get
\begin{equation}\label{38}
\left\Vert y_{m}\left( x\right) -y\left( x\right) \right\Vert _{C_{1-\gamma ;\psi }\left[ a,x_{1}\right] }\rightarrow 0.
\end{equation}

Again, by {\rm Lemma \ref{lema3}}, it follows that 
\begin{eqnarray*}
&&\left\Vert I_{a+}^{\alpha ;\psi }f\left( x,y_{m}\left( x\right) \right)
-I_{a+}^{\alpha ;\psi }f\left( x,y\left( x\right) \right) \right\Vert
_{C_{1-\gamma ;\psi }\left[ a,x_{1}\right] } \\
&\leq &\left\Vert \left( \psi \left( x_{1}\right) -\psi \left( a\right)
\right) ^{1-\gamma }\left( f\left( x,y_{m}\left( x\right) \right) -f\left(
x,y\left( x\right) \right) \right) \right\Vert _{C\left[ a,x_{1}\right] } \\
&&\times \left\Vert \left( \psi \left( x_{1}\right) -\psi \left( a\right)
\right) ^{1-\gamma }I_{a+}^{\alpha ;\psi }\left( \psi \left( x_{1}\right)
-\psi \left( a\right) \right) ^{\gamma -1}\right\Vert _{C\left[ a,x_{1}%
\right] } \\
&\leq &A\frac{\Gamma \left( \gamma \right) }{\Gamma \left( \gamma +\alpha
\right) }\left( \psi \left( x_{1}\right) -\psi \left( a\right) \right)
^{\alpha }\left\Vert y_{m}\left( x\right) -y\left( x\right) \right\Vert
_{C_{1-\gamma ;\psi }\left[ a,x_{1}\right] }
\end{eqnarray*}
and hence by {\rm Eq.(\ref{38})}
\begin{equation}\label{39}
\left\Vert I_{a+}^{\alpha ;\psi }f\left( x,y_{m}\left( x\right) \right) -I_{a+}^{\alpha ;\psi }f\left( x,y\left( x\right) \right) \right\Vert _{C_{1-\gamma ;\psi }\left[ a,x_{1}\right] }\rightarrow 0 
\end{equation}
as $m\rightarrow \infty$.

Using {\rm Eq.(\ref{38})} and {\rm Eq.(\ref{39})}, we conclude that $y\left( x\right) $ is the solution of integral equation {\rm Eq.(\ref{eq21})} in $\ C_{1-\gamma ;\psi }\left[ a,x_{1}\right]$ .

On the other hand, considering two solutions $y\left( x\right) $ and $z\left( x\right) $ of the integral equation {\rm Eq.(\ref{eq21})} on $\left[ a,x_{1}\right] $, we prove that the solution $y\left( x\right) $ is unique. Substituting them into {\rm Eq.(\ref{39})} and using {\rm Lemma \ref{lema1}} with Lipschitz condition {\rm Eq.(\ref{eq10})}, we get
\begin{eqnarray}\label{310}
\left\Vert y\left( x\right) -z\left( x\right) \right\Vert _{_{C_{1-\gamma
;\psi }\left[ a,x_{1}\right] }} &\leq &\left\Vert I_{a+}^{\alpha ;\psi
}\left( \psi \left( x\right) -\psi \left( a\right) \right) ^{\gamma
-1}\right\Vert _{C_{1-\gamma ;\psi }\left[ a,x_{1}\right] }  \notag \\
&&\left\Vert f\left( x,y\left( x\right) \right) -f\left( x,z\left( x\right)
\right) \right\Vert _{C_{1-\gamma ;\psi }\left[ a,x_{1}\right] }  \notag \\
&\leq &A\frac{\Gamma \left( \gamma \right) \left( \psi \left( x\right) -\psi
\left( a\right) \right) ^{\alpha }}{\Gamma \left( \gamma +\alpha \right) }%
\left\Vert y\left( x\right) -z\left( x\right) \right\Vert _{_{C_{1-\gamma
;\psi }\left[ a,x_{1}\right] }}.
\end{eqnarray}

From this it follows that $A\dfrac{\Gamma \left( \gamma \right) \left( \psi \left( x\right) -\psi \left( a\right) \right) ^{\alpha }}{\Gamma \left( \gamma +\alpha \right) }\geq 1,$ contradicting the condition {\rm Eq.(\ref{32})}. Thus there exists $y\left( x\right) =y_{1}\left( x\right) \in C_{1-\gamma ;\psi }\left[ a,x_{1}\right] $ as a unique solution on $\left[ a,x_{1}\right] .$

Now, consider the interval $\left[ x,x_{2}\right] $, where $%
x_{2}=x_{1}+h_{1},$ $h_{1}>0$ such that $x_{2}<b.$ Now the integral equation {\rm Eq.(\ref{310})} takes the form 
\begin{eqnarray}\label{311}
y\left( x\right)  &=&\frac{y_{a}}{\Gamma \left( \gamma \right) }\left( \psi
\left( x\right) -\psi \left( a\right) \right) ^{\gamma -1}+\frac{1}{\Gamma
\left( \alpha \right) }\int_{x_{1}}^{x}\psi ^{\prime }\left( t\right) \left(
\psi \left( x\right) -\psi \left( t\right) \right) ^{\alpha -1}f\left(
t,y\left( t\right) \right) dt+  \notag  \label{311} \\
&&+\frac{1}{\Gamma \left( \alpha \right) }\int_{a}^{x_{1}}\psi ^{\prime
}\left( t\right) \left( \psi \left( x\right) -\psi \left( t\right) \right)
^{\alpha -1}f\left( t,y\left( t\right) \right) dt,
\end{eqnarray}
with $x\in \left[ x_{1},x_{2}\right]$.

As seen previously, $y\left( x\right) $ is uniquely defined on $\left[ a,x_{1}\right] $ and the integral equation {\rm Eq.(\ref{311})} is known and can be written as follows
\begin{equation}\label{312}
y\left( x\right) =y_{0}^{\ast }\left( x\right) +\frac{1}{\Gamma \left(
\alpha \right) }\int_{x_{1}}^{x}\psi ^{\prime }\left( t\right) \left( \psi
\left( x\right) -\psi \left( t\right) \right) ^{\alpha -1}f\left( t,y\left(
t\right) \right) dt,
\end{equation}
with $x\in \left[ x_{1},x_{2}\right] $, where
\begin{equation}\label{313}
y_{0}^{\ast }\left( x\right) =\frac{y_{a}}{\Gamma \left( \gamma \right) }\left( \psi \left( x\right) -\psi \left( a\right) \right) ^{\gamma -1}+\frac{1}{\Gamma \left( \alpha \right) }\int_{a}^{x_{1}}\psi ^{\prime }\left(
t\right) \left( \psi \left( x\right) -\psi \left( t\right) \right) ^{\alpha-1}f\left( t,y\left( t\right) \right) dt
\end{equation}
is a known function. Using the same argument as above, we deduce that there exist a unique solution $y\left( x\right) =y_{2}\left( x\right) \in C_{1-\gamma ;\psi }\left[ x_{1},x_{2}\right] $ on $\left[ x_{1},x_{2}\right]
.$ Taking interval $\left[ x_{2},x_{3}\right] $, where $x_{3}=x_{2}+h_{2}$, $h_{2}>0$ such that $x_{3}<b,$ and repeating the above steps, we obtain a unique solution $y\left( x\right) \in C_{1-\gamma
;\psi }\left[ a,b\right] $, of the integral equation {\rm Eq.(\ref{eq21})} such that $y\left( x\right) =y_{j}\left( x\right) \in C_{1-\gamma ;\psi }\left[ x_{j-1},x_{j}\right] ,$ for $j=1,2,...,l$ and $\ a=x_{0}<x_{2}<\cdot \cdot \cdot <x_{l}=b$. Using differential equation {\rm Eq.(\ref{IVPI})} and Lipschitz condition {\rm Eq.(\ref{eq10})}, we obtain
\begin{eqnarray}\label{314}
\left\Vert ^{H}D_{a+}^{\alpha ,\beta ;\psi }y_{m}\left( x\right) -\text{ } ^{H}D_{a+}^{\alpha ,\beta ;\psi }y\left( x\right) \right\Vert _{C_{1-\gamma ;\psi }\left[ a,b\right] } &=&\left\Vert f\left( x,y_{m}\left( x\right)
\right) -f\left( x,y\left( x\right) \right) \right\Vert _{C_{1-\gamma ;\psi }\left[ a,b\right] }  \notag \\
&\leq &A\left\Vert y_{m}\left( x\right) -y\left( x\right) \right\Vert _{C_{1-\gamma ;\psi }\left[ a,b\right] }.
\end{eqnarray}

Therefore, by {\rm Eq.(\ref{38})} and {\rm Eq.(\ref{314})} implies that $^{H}\mathbb{D}_{a+}^{\alpha ,\beta ;\psi }y\left( x\right) \in C_{1-\gamma ;\psi }\left[ a,b\right] $ and thus we conclude the result.
\end{proof}

\section{Continuous dependence}
Working with initial value problem in which it can models a physical phenomenon, it is desirable that any little perturbation in the initial data does not influence the solution. In this section, we study the data continuous dependence of the Cauchy problem solution for a fractional differential equation using $\psi$-Hilfer derivative via the generalized Gronwall inequality, obtained in Eq.(\ref{jose}).

\begin{theorem}Let $f,\psi\in C([a,b],\mathbb{R})$  two functions such that $\psi$ is increasing and $\psi'(x)\neq 0$, for all $x\in [a,b]$ and $f$ satisfying Lipschitz condition {\rm Eq.(\ref{eq10})} in $\mathbb{R}$. Let $\alpha >0$, $\delta>0$ such that $0<\alpha -\delta <\alpha \leq 1$ and $0 \leq \beta \leq 1$. For $a\leq x\leq h<b$, assume that $y$ is the solution of the initial value problem $(IVP)$ {\rm Eq.(\ref{IVPI})} and $\overset{\ast }{y}$ is the solution of the following IVP
\begin{eqnarray}\label{IVPII}
^{H}\mathbb{D}_{a+}^{\alpha -\delta ,\beta;\psi }\overset{\ast }{y}\left( x\right)
&=&f\left( x,\overset{\ast }{y}\left( x\right) \right) ,\text{ }0<\alpha <1,%
\text{ }0\leq \beta \leq 1 \\
I_{a+}^{1-\gamma -\delta \left( \beta -1\right) ;\psi }\overset{\ast }{y}%
\left( x\right) _{x=a} &=&\text{ }\overset{\ast }{y}_{a},\text{ }\gamma
=\alpha +\beta \left( 1-\alpha \right).
\end{eqnarray}

Then, for $a<x\leq h$,
\begin{equation*}
\left\vert \overset{\ast }{y}\left( x\right) -y\left( x\right) \right\vert
\leq B\left( x\right) +\int_{a}^{x}\left[ \underset{n=1}{\overset{\infty }{%
\sum }}\left( \frac{A}{\Gamma \left( \alpha \right) }\Gamma \left( \alpha
-\delta \right) \right) ^{n}\frac{\left( \psi \left( x\right) -\psi \left(
t\right) \right) ^{n\left( \alpha -\delta \right) -1}}{\Gamma \left( n\left(
\alpha -\delta \right) \right) }B\left( t\right) \right] dt
\end{equation*}
holds, where
\begin{eqnarray}\label{43}
B\left( x\right) &=&\left\vert \frac{\overset{\ast }{y}_{a}\left( \psi
\left( x\right) -\psi \left( a\right) \right) ^{\gamma +\delta \left( \beta
-1\right) -1}}{\Gamma \left( \gamma +\delta \left( \beta -1\right) \right) }-%
\frac{y_{a}\left( \psi \left( x\right) -\psi \left( a\right) \right)
^{\gamma -1}}{\Gamma \left( \gamma \right) }\right\vert  \notag \\
&&+\left\Vert f\right\Vert \left\vert \frac{\left( \psi \left( x\right)
-\psi \left( a\right) \right) ^{\alpha -\delta }}{\Gamma \left( \alpha
-\delta +1\right) }-\frac{\left( \psi \left( x\right) -\psi \left( a\right)
\right) ^{\alpha -\delta }}{\left( \alpha -\delta \right) \Gamma \left(
\alpha \right) }\right\vert  \notag \\
&&+\left\Vert f\right\Vert \left\vert \frac{\left( \psi \left( x\right)
-\psi \left( a\right) \right) ^{\alpha -\delta }}{\left( \alpha -\delta
\right) \Gamma \left( \alpha \right) }-\frac{\left( \psi \left( x\right)
-\psi \left( a\right) \right) ^{\alpha }}{\Gamma \left( \alpha +1\right) }%
\right\vert
\end{eqnarray}
and
\begin{equation*}
\left\Vert f\right\Vert =\underset{x\in \left[ a,b\right] }{\max }\left\vert
f\left( x,y\left( x\right) \right) \right\vert.
\end{equation*}
\end{theorem} 

\begin{proof} IVP's {\rm Eq.(\ref{IVPI})-Eq.(4.2)} and {\rm Eq.(\ref{IVPII})}-{\rm Eq.(5.2)}, have similar integral solutions and are given by
\begin{equation*}
y\left( x\right) =\frac{y_{a}\left( \psi \left( x\right) -\psi \left(
a\right) \right) ^{\gamma -1}}{\Gamma \left( \gamma \right) }+\frac{1}{%
\Gamma \left( \alpha \right) }\int_{a}^{x}\psi ^{\prime }\left( t\right)
\left( \psi \left( x\right) -\psi \left( t\right) \right) ^{\alpha
-1}f\left( t,y\left( t\right) \right) dt
\end{equation*}
and
\begin{equation*}
\overset{\ast }{y}\left( x\right) =\frac{\overset{\ast }{y}_{a}\left( \psi
\left( x\right) -\psi \left( a\right) \right) ^{\gamma +\delta \left( \beta
-1\right) -1}}{\Gamma \left( \gamma +\delta \left( \beta -1\right) \right) }+%
\frac{1}{\Gamma \left( \alpha -\delta \right) }\int_{a}^{x}\psi ^{\prime
}\left( t\right) \left( \psi \left( x\right) -\psi \left( t\right) \right)
^{\alpha -\delta -1}f\left( t,\overset{\ast }{y}\left( t\right) \right) dt
\end{equation*}
respectively. It follows that 
\begin{eqnarray*}
\left\vert \overset{\ast }{y}\left( x\right) -y\left( x\right) \right\vert 
&\leq &\left\vert \frac{\overset{\ast }{y}_{a}\left( \psi \left( x\right)
-\psi \left( a\right) \right) ^{\gamma +\delta \left( \beta -1\right) -1}}{%
\Gamma \left( \gamma +\delta \left( \beta -1\right) \right) }-\frac{%
y_{a}\left( \psi \left( x\right) -\psi \left( a\right) \right) ^{\gamma -1}}{%
\Gamma \left( \gamma \right) }\right\vert +  \notag \\
&&\left\vert \frac{1}{\Gamma \left( \alpha -\delta \right) }\int_{a}^{x}\psi
^{\prime }\left( t\right) \left( \psi \left( x\right) -\psi \left( t\right)
\right) ^{\alpha -\delta -1}f\left( t,\overset{\ast }{y}\left( t\right)
\right) dt\right.   \notag \\
&&-\left. \frac{1}{\Gamma \left( \alpha \right) }\int_{a}^{x}\psi ^{\prime
}\left( t\right) \left( \psi \left( x\right) -\psi \left( t\right) \right)
^{\alpha -1}f\left( t,y\left( t\right) \right) dt\right\vert   \notag \\
&=&\left\vert \frac{\overset{\ast }{y}_{a}\left( \psi \left( x\right) -\psi
\left( a\right) \right) ^{\gamma +\delta \left( \beta -1\right) -1}}{\Gamma
\left( \gamma +\delta \left( \beta -1\right) \right) }-\frac{y_{a}\left(
\psi \left( x\right) -\psi \left( a\right) \right) ^{\gamma -1}}{\Gamma
\left( \gamma \right) }\right\vert +  \notag \\
&&\left\vert \int_{a}^{x}\psi ^{\prime }\left( t\right) \left[ \frac{\left(
\psi \left( x\right) -\psi \left( t\right) \right) ^{\alpha -\delta -1}}{%
\Gamma \left( \alpha -\delta \right) }-\frac{\left( \psi \left( x\right)
-\psi \left( t\right) \right) ^{\alpha -\delta -1}}{\Gamma \left( \alpha
\right) }\right] f\left( t,\overset{\ast }{y}\left( t\right) \right)
dt\right. +  \notag \\
&&\frac{1}{\Gamma \left( \alpha \right) }\int_{a}^{x}\psi ^{\prime }\left(
t\right) \left( \psi \left( x\right) -\psi \left( t\right) \right) ^{\alpha
-\delta -1}\left[ f\left( t,\overset{\ast }{y}\left( t\right) \right)
-f\left( t,y\left( t\right) \right) \right] dt  \notag \\
&&+\left. \frac{1}{\Gamma \left( \alpha \right) }\int_{a}^{x}\psi ^{\prime
}\left( t\right) \left[ \left( \psi \left( x\right) -\psi \left( t\right)
\right) ^{\alpha -\delta -1}-\left( \psi \left( x\right) -\psi \left(
t\right) \right) ^{\alpha -1}\right] f\left( t,y\left( t\right) \right)
dt\right\vert   \notag \\
&\leq &\left\vert \frac{\overset{\ast }{y}_{a}\left( \psi \left( x\right)
-\psi \left( a\right) \right) ^{\gamma +\delta \left( \beta -1\right) -1}}{%
\Gamma \left( \gamma +\delta \left( \beta -1\right) \right) }-\frac{%
y_{a}\left( \psi \left( x\right) -\psi \left( a\right) \right) ^{\gamma -1}}{%
\Gamma \left( \gamma \right) }\right\vert +  \notag \\
&&\left\Vert f\right\Vert \left\vert \int_{a}^{x}\psi ^{\prime }\left(
t\right) \left[ \frac{\left( \psi \left( x\right) -\psi \left( t\right)
\right) ^{\alpha -\delta -1}}{\Gamma \left( \alpha -\delta \right) }-\frac{%
\left( \psi \left( x\right) -\psi \left( t\right) \right) ^{\alpha -\delta
-1}}{\Gamma \left( \alpha \right) }\right] dt\right\vert   \notag \\
&&+\frac{A}{\Gamma \left( \alpha \right) }\int_{a}^{x}\psi ^{\prime }\left(
t\right) \left( \psi \left( x\right) -\psi \left( t\right) \right) ^{\alpha
-\delta -1}\left\vert \overset{\ast }{y}\left( t\right) -y\left( t\right)
\right\vert dt  \notag \\
&&+\left\Vert f\right\Vert \left\vert \int_{a}^{x}\psi ^{\prime }\left(
t\right) \left[ \frac{\left( \psi \left( x\right) -\psi \left( t\right)
\right) ^{\alpha -\delta -1}}{\Gamma \left( \alpha \right) }-\frac{\left(
\psi \left( x\right) -\psi \left( t\right) \right) ^{\alpha -1}}{\Gamma
\left( \alpha \right) }\right] dt\right\vert   \notag \\
&=&B\left( x\right) +\frac{A}{\Gamma \left( \alpha \right) }\int_{a}^{x}\psi
^{\prime }\left( t\right) \left( \psi \left( x\right) -\psi \left( t\right)
\right) ^{\alpha -\delta -1}\left\vert \overset{\ast }{y}\left( t\right)
-y\left( t\right) \right\vert dt
\end{eqnarray*}
where $B(x)$ is defined by {\rm Eq.(\ref{43})}. Applying the Gronwall inequality {\rm Eq.(\ref{jose})}, we conclude that
\begin{equation*}
\left\vert \overset{\ast }{y}\left( x\right) -y\left( x\right) \right\vert
\leq B\left( x\right) +\int_{a}^{x}\left[ \overset{\infty }{\underset{n=1}{%
\sum }}\left( \frac{A\Gamma \left( \alpha -\delta \right) }{\Gamma \left(
\alpha \right) }\right) ^{n}\frac{\psi ^{\prime }\left( t\right) \left( \psi
\left( x\right) -\psi \left( t\right) \right) ^{n\left( \alpha -\delta
\right) -1}}{\Gamma \left( n\left( \alpha -\delta \right) \right) }B\left(
t\right) \right] dt.
\end{equation*}
\end{proof}

Next, we consider the fractional differential equation {\rm Eq.(\ref{IVPI})} with small change in the initial condition Eq.(4.2)  
\begin{equation*}
I_{a+}^{1-\gamma ;\psi }y\left( x\right) _{x=a}=y_{a}+\varepsilon \text{, }%
\gamma =\alpha +\beta \left( 1-\alpha \right),
\end{equation*}
where $\varepsilon $ is an arbitrary positive constant. We state and prove the result as
follows.

\begin{theorem} Suppose that assumptions of {\rm Theorem \ref{teo3}} hold. Suppose $y\left(
x\right) $ and $\overset{\ast }{y}\left( x\right) $ are solutions of IVP
{\rm Eq.(\ref{IVPI})}-{\rm Eq.(4.2)} and {\rm Eq.(\ref{IVPII})}-{\rm Eq.(5.2)} respectively. Then
\begin{equation*}
\left\vert y\left( x\right) -\overset{\ast }{y}\left( x\right) \right\vert
\leq \left\vert \varepsilon \right\vert \left( \psi \left( x\right) -\psi
\left( a\right) \right) ^{\gamma -1}\mathbb{E}_{\alpha ,\gamma }\left[ A\left( \psi
\left( x\right) -\psi \left( a\right) \right) ^{\alpha }\right] ,\text{ }%
x\left( a,b\right] 
\end{equation*}
holds, where $\mathbb{E}_{\alpha ,\gamma }\left( z\right) =\underset{k=0}{\overset{%
\infty }{\sum }}\dfrac{z^{k}}{\Gamma \left( \alpha k+\gamma \right) }$ is the
Mittag-Leffler function with $Re(\alpha)>0$.
\end{theorem}

\begin{proof} By {\rm Theorem \ref{teo3}}, we have $y\left( x\right) =\underset{
m\rightarrow \infty }{\lim }y_{m}\left( x\right) $ with $y_{0}\left( x\right) $ as defined in {\rm Eq.(\ref{IVPII})} and {\rm Eq.(5.2)}, respectively. Clearly, we can write $\overset{\ast }{y}\left( x\right) =\underset{m\rightarrow \infty }{\lim }$ $\overset{\ast }{y_{m}}\left( x\right) $ and
\begin{equation}\label{46}
\overset{\ast }{y_{0}}\left( x\right) =\frac{\left( y_{a}+\varepsilon
\right) }{\Gamma \left( \gamma \right) }\left( \psi \left( x\right) -\psi
\left( a\right) \right) ^{\gamma -1},
\end{equation}
and
\begin{equation*}
\overset{\ast }{y_{m}}\left( x\right) =\overset{\ast }{y_{0}}\left( x\right)
+\frac{1}{\Gamma \left( \alpha \right) }\int_{a}^{x}\psi ^{\prime }\left(
t\right) \left( \psi \left( x\right) -\psi \left( t\right) \right) ^{\alpha
-1}f\left( t,\overset{\ast }{y}_{m-1}\left( t\right) \right) dt.
\end{equation*}

From {\rm Eq.(\ref{eq21})} and {\rm Eq.(\ref{46})}, we have
\begin{eqnarray}\label{48}
\left\vert y_{0}\left( x\right) -\overset{\ast }{y_{0}}\left( x\right)
\right\vert  &=&\left\vert \frac{y_{a}}{\Gamma \left( \gamma \right) }\left(
\psi \left( x\right) -\psi \left( a\right) \right) ^{\gamma -1}-\frac{\left(
y_{a}+\varepsilon \right) }{\Gamma \left( \gamma \right) }\left( \psi \left(
x\right) -\psi \left( a\right) \right) ^{\gamma -1}\right\vert \notag  \\
&\leq &\left\vert \varepsilon \right\vert \frac{\left( \psi \left( x\right)
-\psi \left( a\right) \right) ^{\gamma -1}}{\Gamma \left( \gamma \right) }
\end{eqnarray}

Using relations {\rm Eq.(\ref{34})} and {\rm Eq.(\ref{46})}, the Lipschitz condition
{\rm Eq.(\ref{eq10})} and the inequality {\rm Eq.(\ref{48})}, we get
\begin{eqnarray*}
&&\left\vert y_{1}\left( x\right) -\overset{\ast }{y_{1}}\left( x\right)
\right\vert   \notag \\
&=&\left\vert \frac{\varepsilon \left( \psi \left( x\right) -\psi \left(
a\right) \right) ^{\gamma -1}}{\Gamma \left( \gamma \right) }+\frac{1}{
\Gamma \left( \alpha \right) }\int_{a}^{x}\psi ^{\prime }\left( t\right)
\left( \psi \left( x\right) -\psi \left( t\right) \right) ^{\alpha -1}\left(
f\left( t,y_{0}\left( t\right) \right) -f\left( t,\overset{\ast }{y}
_{0}\left( t\right) \right) \right) dt\right\vert   \notag \\
&\leq &\left\vert \varepsilon \right\vert \frac{\left( \psi \left( x\right)
-\psi \left( a\right) \right) ^{\gamma -1}}{\Gamma \left( \gamma \right) }+
\frac{A}{\Gamma \left( \alpha \right) }\int_{a}^{x}\psi ^{\prime }\left(
t\right) \left( \psi \left( x\right) -\psi \left( t\right) \right) ^{\alpha
-1}\left\vert y_{0}\left( t\right) -\overset{\ast }{y}_{0}\left( t\right)
\right\vert dt  \notag \\
&\leq &\left\vert \varepsilon \right\vert \frac{\left( \psi \left( x\right)
-\psi \left( a\right) \right) ^{\gamma -1}}{\Gamma \left( \gamma \right) }+
\frac{A\left\vert \varepsilon \right\vert \left( \psi \left( x\right) -\psi
\left( a\right) \right) ^{\gamma -1}}{\Gamma \left( \alpha \right) \Gamma
\left( \gamma \right) }\int_{a}^{x}\psi ^{\prime }\left( t\right) \left(
\psi \left( x\right) -\psi \left( t\right) \right) ^{\alpha -1}dt  \notag \\
&=&\left\vert \varepsilon \right\vert \frac{\left( \psi \left( x\right)
-\psi \left( a\right) \right) ^{\gamma -1}}{\Gamma \left( \gamma \right) }
+A\left\vert \varepsilon \right\vert \frac{\left( \psi \left( x\right) -\psi
\left( a\right) \right) ^{\gamma +\alpha -1}}{\Gamma \left( \gamma +\alpha
\right) }.
\end{eqnarray*}

From this it follows that
\begin{equation*}
\left\vert y_{1}\left( x\right) -\overset{\ast }{y_{1}}\left( x\right)
\right\vert \leq \left\vert \varepsilon \right\vert \left( \psi \left(
x\right) -\psi \left( a\right) \right) ^{\gamma -1}\overset{1}{\underset{j=0}%
{\sum }}A^{j}\frac{\left( \psi \left( x\right) -\psi \left( a\right) \right)
^{\alpha j}}{\Gamma \left( \gamma +\alpha j\right) }.
\end{equation*}

On the other hand, we have
\begin{eqnarray*}
&&\left\vert y_{2}\left( x\right) -\overset{\ast }{y_{2}}\left( x\right)
\right\vert  \\
&=&\left\vert \frac{\varepsilon \left( \psi \left( x\right) -\psi \left(
a\right) \right) ^{\gamma -1}}{\Gamma \left( \gamma \right) }+\frac{1}{%
\Gamma \left( \alpha \right) }\int_{a}^{x}\psi ^{\prime }\left( t\right)
\left( \psi \left( x\right) -\psi \left( t\right) \right) ^{\alpha -1}\left(
f\left( t,y_{1}\left( t\right) \right) -f\left( t,\overset{\ast }{y}%
_{1}\left( t\right) \right) \right) dt\right\vert  \\
&\leq &\left\vert \varepsilon \right\vert \frac{\left( \psi \left( x\right)
-\psi \left( a\right) \right) ^{\gamma -1}}{\Gamma \left( \gamma \right) }+%
\frac{A}{\Gamma \left( \alpha \right) }\int_{a}^{x}\psi ^{\prime }\left(
t\right) \left( \psi \left( x\right) -\psi \left( t\right) \right) ^{\alpha
-1}\left\vert y_{1}\left( t\right) -\overset{\ast }{y}_{1}\left( t\right)
\right\vert dt \\
&\leq &\left\vert \varepsilon \right\vert \frac{\left( \psi \left( x\right)
-\psi \left( a\right) \right) ^{\gamma -1}}{\Gamma \left( \gamma \right) }+%
\frac{A\left\vert \varepsilon \right\vert }{\Gamma \left( \alpha \right) }%
\int_{a}^{x}\psi ^{\prime }\left( t\right) \left( \psi \left( x\right) -\psi
\left( t\right) \right) ^{\alpha -1}\left( \psi \left( t\right) -\psi \left(
a\right) \right) ^{\gamma -1} \\
&&\times \overset{1}{\underset{j=0}{\sum }}A^{j}\frac{\left( \psi \left(
t\right) -\psi \left( a\right) \right) ^{\alpha j}}{\Gamma \left( \gamma
+\alpha j\right) }dt \\
&=&\left\vert \varepsilon \right\vert \frac{\left( \psi \left( x\right)
-\psi \left( a\right) \right) ^{\gamma -1}}{\Gamma \left( \gamma \right) }+%
\frac{A\left\vert \varepsilon \right\vert }{\Gamma \left( \alpha \right) }%
\overset{1}{\underset{j=0}{\sum }}\frac{A^{j}}{\Gamma \left( \gamma +\alpha
j\right) } \\
&&\times \left\{ \int_{a}^{x}\psi ^{\prime }\left( t\right) \left( \psi
\left( x\right) -\psi \left( t\right) \right) ^{\alpha -1}\left( \psi \left(
t\right) -\psi \left( a\right) \right) ^{\gamma +\alpha j-1}dt\right\}  \\
&\leq &\left\vert \varepsilon \right\vert \frac{\left( \psi \left( x\right)
-\psi \left( a\right) \right) ^{\gamma -1}}{\Gamma \left( \gamma \right) }+%
\frac{A\left\vert \varepsilon \right\vert \left( \psi \left( x\right) -\psi
\left( a\right) \right) ^{\gamma -1}}{\Gamma \left( \alpha \right) \Gamma
\left( \gamma \right) }\int_{a}^{x}\psi ^{\prime }\left( t\right) \left(
\psi \left( x\right) -\psi \left( t\right) \right) ^{\alpha -1}dt+ \\
&&+\frac{A^{2}\left\vert \varepsilon \right\vert \left( \psi \left( x\right)
-\psi \left( a\right) \right) ^{\gamma +\alpha -1}}{\Gamma \left( \alpha
\right) \Gamma \left( \alpha +\gamma \right) }\int_{a}^{x}\psi ^{\prime
}\left( t\right) \left( \psi \left( x\right) -\psi \left( t\right) \right)
^{\alpha -1}dt \\
&\leq &\left\vert \varepsilon \right\vert \frac{\left( \psi \left( x\right)
-\psi \left( a\right) \right) ^{\gamma -1}}{\Gamma \left( \gamma \right) }+%
\frac{A\left\vert \varepsilon \right\vert \left( \psi \left( x\right) -\psi
\left( a\right) \right) ^{\gamma -1}\left( \psi \left( x\right) -\psi \left(
a\right) \right) ^{\alpha }}{\Gamma \left( \alpha +\gamma \right) } \\
&&+\frac{A^{2}\left\vert \varepsilon \right\vert \left( \psi \left( x\right)
-\psi \left( a\right) \right) ^{\gamma -1}\left( \psi \left( x\right) -\psi
\left( a\right) \right) ^{2\alpha }}{\Gamma \left( 2\alpha +\gamma \right) }
\\
&=&\left\vert \varepsilon \right\vert \left( \psi \left( x\right) -\psi
\left( a\right) \right) ^{\gamma -1}\underset{j=0}{\overset{2}{\sum }}A^{j}%
\frac{\left( \psi \left( x\right) -\psi \left( a\right) \right) ^{\alpha j}}{%
\Gamma \left( \alpha j+\gamma \right) }.
\end{eqnarray*}

Using the induction, we get
\begin{equation} \label{410}
\left\vert y_{m}\left( x\right) -\overset{\ast }{y_{m}}\left( x\right)
\right\vert \leq \left\vert \varepsilon \right\vert \left( \psi \left(
x\right) -\psi \left( a\right) \right) ^{\gamma -1}\overset{m}{\underset{j=0}%
{\sum }}A^{j}\frac{\left( \psi \left( x\right) -\psi \left( a\right) \right)
^{\alpha j}}{\Gamma \left( \alpha j+\gamma \right) }.
\end{equation}

Taking the limit $m\rightarrow \infty $ in {\rm Eq.(\ref{410})}, we conclude that
\begin{eqnarray*}
\left\vert y_{m}\left( x\right) -\overset{\ast }{y_{m}}\left( x\right) \right\vert  &\leq &\left\vert \varepsilon \right\vert \left( \psi \left( x\right) -\psi \left( a\right) \right) ^{\gamma -1}\overset{m}{\underset{j=0}
{\sum }}A^{j}\frac{\left( \psi \left( x\right) -\psi \left( a\right) \right) ^{\alpha j}}{\Gamma \left( \alpha j+\gamma \right) }  \notag \\
&=&\left\vert \varepsilon \right\vert \left( \psi \left( x\right) -\psi \left( a\right) \right) ^{\gamma -1}\mathbb{E}_{\alpha ,\gamma }\left( A\left( \psi \left( x\right) -\psi \left( a\right) \right) ^{\alpha }\right) .
\end{eqnarray*}
\end{proof}

\section{Concluding remarks}

In this paper, combining the fractional integral with respect to another $\psi$ function and the classical Gronwall inequality \cite{DRA}, we proposed a generalized Gronwall inequality. For this inequality, we discussed the existence and uniqueness of solutions of the Cauchy-type problem by means of $\psi$-Hilfer fractional derivative recently introduced \cite{JEM1}. On the other hand, as application, the continuous dependence of the Cauchy-type problem on data was studied via generalized Gronwall inequality.

\section*{Acknowledgment}
The authors thank Espa\c{c}o da Escrita, Coordenadoria Geral da Universidade - Unicamp, for the language services provided.

\bibliography{ref}
\bibliographystyle{plain}

\end{document}